\documentclass[11pt]{amsart}

\usepackage{amsmath,amssymb,amsthm,mathtools}
\usepackage{epsfig}
\usepackage{graphics}
\usepackage{color}

\newtheorem{theorem}{Theorem}[section]
\newtheorem{proposition}[theorem]{Proposition}
\newtheorem{lemma}[theorem]{Lemma}
\newtheorem{corollary}[theorem]{Corollary}
\newtheorem{definition}[theorem]{Definition}
\newtheorem{remark}[theorem]{Remark}
\newtheorem{example}[theorem]{Example}

\newcommand{\MF}{\operatorname{MF}}

\newcommand{\Fac}{\operatorname{Fac}}
\newcommand{\VVert}{\operatorname{Vert}}
\newcommand{\Edge}{\operatorname{Edge}}
\newcommand{\bbC}{\mathbb C}
\newcommand{\bbQ}{\mathbb Q}
\newcommand{\bbR}{\mathbb R}
\newcommand{\bbZ}{\mathbb Z}

\newcommand{\card}[1]{\lvert #1\rvert}

\numberwithin{equation}{section}

\begin{document}
\title[Homogeneous Milnor fibers and Kato--Matsumoto bounds]{Homogeneous Milnor fibers and Kato--Matsumoto bounds via simplicial multiwedges}

\author[Masaharu Ishikawa]{Masaharu Ishikawa}
\address{Faculty of Economics, Keio University, 4-1-1, Hiyoshi, Kouhoku, Yokohama, Kanagawa 223-8521, Japan}
\email{ishikawa@keio.jp}

\author[Tat-Thang Nguyen]{Tat-Thang Nguyen}
\address{Institute of Mathematics, Vietnam Academy of Science and Technology, 
18 Hoang Quoc Viet Road, Nghia Do ward, 10072 Hanoi, Vietnam}
\email{ntthang@math.ac.vn}



\begin{abstract}
For every $n\geq3$ and $s\geq2$, we construct a homogeneous polynomial of degree $n(n+1)/2$ whose Milnor fiber is exactly $2s$-connected and whose rational cohomology contains a strictly defined nontrivial $n$-fold Massey product on classes of degree $2s+1$, implying that the Milnor fiber is non-formal, while attaining the Kato--Matsumoto connectivity bound.
Our construction is based on the simplicial multiwedges of the nerve complexes of simple polytopes introduced by Limonchenko, combined with Suciu's realization of 
weighted homogeneous Milnor fibers.
We thereby answer two problems posed by Suciu.
\end{abstract}

\maketitle

\section{Introduction}

This paper concerns the non-formality and connectivity of Milnor fibers associated with singularities having trivial geometric monodromy. 
The formality question for Milnor fibers of hypersurface singularities was raised by Papadima and Suciu~\cite{PS09}. 
The first non-formal Milnor fiber was found by Zuber~\cite{Zub10} using the Ceva arrangement.
This Milnor fiber is not simply connected. See also~\cite{Suciu2026-2} for further developments in this direction.
In~\cite{FdB},
Fern\'{a}ndez de Bobadilla constructed simply connected, non-formal Milnor fibers with trivial geometric monodromy, using nontrivial triple Massey products arising from the work of Denham and Suciu~\cite{DS07} on moment-angle complexes.
Recently, Suciu proved in~\cite{Suciu2026} that there exist weighted homogeneous polynomials with trivial geometric monodromy whose Milnor fibers are non-formal and admit nontrivial Massey products of arbitrarily high order. His construction combines Fern\'{a}ndez de Bobadilla's realization theorem of Milnor fibers with the systematic construction of moment-angle complexes 
due to Grbi\'{c} and Linton in~\cite{GL21}.
Moreover, he showed that the connectivity of the corresponding Milnor fibers attains the Kato--Matsumoto bound.

Another construction relevant to our purposes is due to Limonchenko, who constructed highly connected moment-angle complexes with nontrivial higher Massey products.
In~\cite[Definition~3.5]{Limonchenko2019}, he introduced a simple polytope $P(n$) having $n(n+1)/2-1$ facets, obtained from the $n$-dimensional cube by truncating suitable $n(n-1)/2-1$ codimension-two faces,
and studied the simplicial multiwedge $K_{P(n)}(J)$ (in other words, $J$-construction) of 
the nerve complex $K_{P(n)}$ of $P(n)$ to obtain a highly connected complex with a strictly defined nontrivial $n$-fold Massey product,
where
\begin{equation}\label{eq1-1}
   J
   =\bigl(
      \underbrace{s,\ldots,s}_{n},
      \underbrace{1,\ldots,1}_{n},
      j_{2n+1,n,s},\ldots,j_{m,n,s}
     \bigr)
\end{equation}
is an $m$-tuple of positive integers such that $j_{p,n,s}=1$ for $p>2n$ if $s=1$ and 
$|I|\ge s+1$ for every minimal non-face $I$ of $K_{P(n)}(J)$.

Applying the argument of Suciu in~\cite{Suciu2026} to the simplicial multiwedge $K_{P(n)}(J)$ of Limonchenko,
as in the following theorem,
we obtain a homogeneous polynomial whose Milnor fiber is highly connected and has a nontrivial Massey product.
 
\begin{theorem}\label{thm1}
For every $n\ge 3$ and $J$ in~\eqref{eq1-1} with $s\geq 1$,
there exists a homogeneous polynomial $f_{n,J}:\bbC^N\to \bbC$ of degree $n(n+1)/2$
with trivial geometric monodromy such that its Milnor fiber $F_{n,J}=f_{n,J}^{-1}(1)$ has the following properties:
\begin{enumerate}
\item $F_{n,J}$ is exactly $2s$-connected;
\item The rational cohomology of $F_{n,J}$ contains a strictly defined nontrivial $n$-fold Massey product on classes of degree $2s+1$, landing in degree $2ns+2$. In particular, $F_{n,J}$ is non-formal.
\end{enumerate}
Moreover, the Kato--Matsumoto bound is sharp if $s\geq 2$ and $j_{2n+1,n,s},\ldots,j_{m,n,s}$ are sufficiently large
\end{theorem}

This theorem resolves the question posed by Suciu in~\cite[Problem~6.1]{Suciu2026},
asking  whether his weighted homogeneous realization theorem can be realized by homogeneous polynomials.  
His Remark~5.13 also identifies an obstruction to obtaining such a realization from the particular Grbi\'c--Linton family used there, since the associated simplicial complexes are not pure.  
We overcome this obstruction by using the simplicial multiwedges introduced by Limonchenko,
which lead to a different construction and yield the desired homogeneous realization.

The same realization can be applied to the generalized moment-angle complexes of $K_{P(n)}$,
which immediately gives an answer to~\cite[Problem~6.6]{Suciu2026}
with 
retaining
all the properties in Theorem~\ref{thm1},
see Corollary~\ref{cor:uniform-realization}.

This paper is organized as follows.
In Section~2, we briefly review known results on simplicial complexes and moment-angle complexes, simplicial multiwedges, and Suciu's results on simplicial complexes and Milnor fibers, including the Kato--Matsumoto bound. In Section~3, we discuss purity under simplicial multiwedges.
Section~4 is devoted to estimating the Kato--Matsumoto bound for simplicial multiwedges.
In Section~5, we introduce Limonchenko's simplicial multiwedge and prove Theorem~\ref{thm1}. 
The answers to Suciu's problems are given at the end of the section.

The second author would like to thank  the International Centre of Research and Postgraduate Training in Mathematics (ICRTM-VAST) for partially supporting the travel to Keio University, and Keio University for hospitality. 
The first author is supported by JSPS KAKENHI Grant Numbers JP23K03098 and JP23H00081.



\section{Preliminaries}

\subsection{Simplicial complexes and moment-angle complexes}

Let $K$ be a simplicial complex of dimension $n-1$ on the set $[m]=\{1,\ldots,m\}$.
In what follows we assume that there are no ghost vertices in $K$, that is, $\{i\}\in K$ for every $i=1,\ldots,m$.
Let $\VVert(K)$ and $\Fac(K)$ denote the sets of vertices and facets of $K$, respectively.
The complex $K$ is said to be {\it pure} if all its facets have the same cardinality.  Thus, if $\dim K=n-1$, purity means that every facet has exactly $n$ vertices.
Throughout the paper, for a set $S$, $|S|$ and $\# S$ denote the cardinality of $S$.

A {\it minimal non-face} $I$ of $K$ is a subset of $\VVert(K)$ that is not a simplex of $K$ but any proper subset of $I$ is a simplex of $K$.
Let $\MF(K)$ denote the set of minimal non-faces of $K$.
Note that  $\VVert(K)$ and $\MF(K)$ determine $K$.
A pair $\{u,v\}$ of vertices of $K$ is called a {\it missing-edge} of $K$ if $\{u,v\}$ is not an edge of $K$.

For a family of pairs of topological spaces  $\{(X_i,A_i)\}_{i=1}^m$ with $A_i\subset X_i$, 
the topological space defined by
\[
   \mathcal Z_K(\underline X,\underline A)
   =\bigcup_{\sigma\in K}\prod_{i=1}^m Y_i(\sigma),
   \qquad
   Y_i(\sigma)=
   \begin{cases}
      X_i,& i\in\sigma,\\
      A_i,& i\notin\sigma
   \end{cases}
\]
is called the {\it polyhedral product} of $\{(X_i,A_i)\}_{i=1}^m$ over $K$.
When all pairs $(X_i,A_i)$ are equal to a fixed pair $(X,A)$, we simply write $\mathcal Z_K(X,A)$ or $(X,A)^K$ for the polyhedral product.
The ordinary {\it moment-angle complex} is given by $\mathcal Z_K=(D^2,S^1)^K$.
For an $m$-tuple $(j_1,\ldots,j_m)$ of positive integers, 
$\mathcal Z_K\bigl((D^{2j_i},S^{2j_i-1})_{i=1}^m\bigr)$
is called the {\it generalized moment-angle complex} of $(D^{2j_i},S^{2j_i-1})_{i=1}^m$ over $K$.
We refer to~\cite{BBCG2010, BBCG2015, BP15, Limonchenko2019} for these constructions and their basic properties.

For a simplicial complex $K$, put
\[
   \nu(K)=\min\{\card{I} \mid I\in\MF(K)\}.
\]
The following proposition follows by combining the results in~\cite{ZZ93} and~\cite{BP15},
see~\cite[Theorem~3.5]{Suciu2026} together with the explanations given around it.
For the convenience of the reader, we include a brief proof.

\begin{proposition}[\cite{ZZ93,BP15}]\label{prop:connectivity}
If $K$ is not a simplex, then $\mathcal Z_K$ is exactly $(2\nu-2)$-connected, where $\nu=\nu(K)$.
More precisely, a minimal non-face $I$ of cardinality $\nu$ contributes a nonzero class in $H^{2\nu-1}(\mathcal Z_K;\bbZ)$.
\end{proposition}

\begin{proof}
By \cite[Theorem~3.5]{Suciu2026}, since every minimal non-face of
$K$ has cardinality at least $\nu$, the moment-angle complex
$\mathcal Z_K$ is $(2\nu-2)$-connected.

Let $I\in\MF(K)$ with $|I|=\nu$. Since $I$ is a minimal non-face,
every proper subset of $I$ is a face of $K$, whereas $I$ itself is not.
Hence the full subcomplex $K_I$ on $I$ is the boundary of a $(\nu-1)$-simplex, and therefore
$\widetilde H^{\nu-2}(K_I;\mathbb Z)\cong\mathbb Z$.
By Hochster's decomposition~\cite{Hoc77} (cf.~\cite{BP02, Bas02}), 
this group occurs as a direct summand of
\[
H^{(\nu-2)+\nu+1}(\mathcal Z_K;\mathbb Z)
=
H^{2\nu-1}(\mathcal Z_K;\mathbb Z).
\]
Thus $H^{2\nu-1}(\mathcal Z_K;\mathbb Z)\neq0$, so
$\mathcal Z_K$ is not $(2\nu-1)$-connected.
\end{proof}

\subsection{The simplicial multiwedge}

Let $K$ be a simplicial complex on the vertex set $[m]=\{1,\ldots,m\}$,
and $J=(j_1,\ldots,j_m)\in\bbZ_{>0}^m$ be an $m$-tuple of positive integers,
where $\bbZ_{>0}=\{k\in \bbZ\mid k>0\}$.
For each $i=1,\ldots, m$, 
let
\[
   B_i=\{i1,\ldots,ij_i\}
\]
be a set of $j_i$ new vertices replacing the vertex $i$. 
We call $B_i$ a {\it block} corresponding to the vertex $i$.
The {\it simplicial multiwedge} $K(J)$ of $K$ with $J$ is the simplicial complex on the vertex set
\[
   V(J)=B_1\sqcup\cdots\sqcup B_m
\]
whose minimal non-faces are given by
\[
   I(J)=\bigcup_{i\in I} B_i,
   \qquad I\in\MF(K).
\]
In other words, each vertex $i$ of $K$ is replaced by $j_i$ vertices, and a minimal non-face $I$ is replaced by the union of the corresponding sets $B_i$.
This construction was introduced by Bahri, Bendersky, Cohen and Gitler in~\cite{BBCG2015}.  It is also called the {\it $J$-construction} or {\it simplicial wedge construction}.

\begin{example}
Let $K$ be a simplicial complex on $[2]=\{1, 2\}$ whose simplicies are the two vertices $\{1\}$ and $\{2\}$.
The edge $\{1,2\}$ is a minimal non-face of $K$.
For $J=(3,1)$, $B_1=\{11,12,13\}$ and $B_2=\{21\}$.
Since $\MF(K)=\{\{1,2\}\}$, we have $I(J)=\{11, 12, 13, 21\}$.
Therefore, $K(J)$ is a simplicial complex on the four vertices $\{11,12,13,21\}$
with $\MF(K(J))=\{\{11,12,13,21\}\}$, which is the boundary of the $3$-simplex on the vertex set $\{11,12,13,21\}$.
\end{example}

The following homeomorphism between generalized moment-angle complexes over $K$ and
ordinary moment-angle complexes over $K(J)$ will be used to prove 
Corollary~\ref{thm:general-realization}, which in turn will be used to prove Corollary~\ref{cor:uniform-realization} at the end of the paper.

\begin{theorem}[{\cite[Theorem~7.5]{BBCG2015}}]\label{thm:exponentiation}
For every simplicial complex $K$ on the set $[m]$ and every $J=(j_1,\ldots,j_m)\in\bbZ_{>0}^m$, 
$\mathcal Z_K\bigl((D^{2j_i},S^{2j_i-1})_{i=1}^m\bigr)$ and $\mathcal Z_{K(J)}(D^2,S^1)$ are equivariantly  homeomorphic with respect to the respective
torus actions.
\end{theorem}

\subsection{Formality and Massey products}\label{sec2-3}

A topological space $X$ is said to be {\it formal} over $\bbQ$ if its rational homotopy type is determined by its cohomology ring.
Brieskorn proved in~\cite{Bri73} that the complement of a hyperplane arrangement is formal.
On the other hand, there are Milnor fibers of hypersurface singularities that are non-formal, as explained in the introduction. In such cases, the topology of the Milnor fiber provides richer information 
than its cohomology ring. 
It is therefore important to understand when non-formality occurs.

A Massey product of a topological space $X$ is a higher-order cohomology operation defined when certain lower-order cup products vanish. More precisely, given cohomology classes $\alpha_i\in H^{p_i}(X;\mathbb{Q})$, the $n$-fold Massey product $\langle\alpha_1,\ldots,\alpha_n\rangle$ is defined under suitable vanishing conditions and is a subset of
$H^{p_1+\cdots+p_n-(n-2)}(X;\mathbb{Q})$.
In general, a Massey product is not a single cohomology class but a set of classes due to an indeterminacy. A Massey product is called {\it nontrivial} if it does not contain $0$.
Non-trivial Massey products provide obstructions to formality.
A Massey product is called {\it strictly defined} if it consists of a single cohomology class, independently of the choice of a defining system. 
Strictly defined Massey products are not necessary for detecting non-formality, but they provide a particularly clear obstruction because their value is independent of the choice of a defining system.
See~\cite{Limonchenko2019} for details.

Denham and Suciu~\cite{DS07} studied the formality of moment-angle complexes and constructed simplicial complexes $K$ for which the associated moment-angle complex $\mathcal Z_K$ admits a nontrivial triple Massey product in rational cohomology. In particular, they exhibited infinitely many triangulations $K$ of $S^2$ for which $\mathcal Z_K$ is non-formal.
Fern\'{a}ndez de Bobadilla used this method to construct a non-formal Milnor fiber~\cite{FdB}.

The vanishing of all triple Massey products does not imply formality in general. 
Thus, higher Massey products may provide further obstructions to formality that cannot be detected by triple Massey products. It is therefore natural to investigate higher Massey products in the study of the formality of Milnor fibers. 
As explained in the next subsection, Suciu~\cite{Suciu2026} constructed weighted homogeneous Milnor fibers with nontrivial $n$-fold Massey products for every $n\geq 3$ and arbitrarily high connectivity.

\subsection{Suciu's polynomial associated with a simplicial complex}

We recall Suciu's construction of a polynomial associated with a simplicial complex~\cite{Suciu2026},
which is based on a construction of Fern\'andez de Bobadilla.

Let $K$ be a simplicial complex on a vertex set of cardinality $m$, with facets $F_1,\ldots,F_r$.  
To each facet $F_j$, associate the square-free monomial
\[
   g_{F_j}(x_1,\ldots, x_m)=\prod_{i\notin F_j}x_i.
\]
Consider the polynomial
\begin{equation}\label{eq:Suciu-poly}
   \Phi_K(x_1,\ldots, x_m, y_1,\ldots, y_r)=\sum_{j=1}^r y_j g_{F_j}(x_1,\ldots,x_m),
\end{equation}
where $y_j$ is a new variable corresponding to the facet $F_j$.
This construction is a specialization of a construction of Fern\'andez de Bobadilla~\cite{FdB}. 
The polynomial  $\Phi_K$ satisfies the following properties.

\begin{theorem}[{\cite[Corollary~3.2]{Suciu2026}}]\label{thm:Suciu-realization}
The polynomial $\Phi_K:\bbC^{m+r}\to\bbC$ is weighted homogeneous. 
Its Milnor fiber $\Phi_K^{-1}(1)$
is homotopy equivalent to $\mathcal Z_K$, and the geometric monodromy of $\Phi_K$ is trivial. 
Moreover, $\Phi_K$ is homogeneous if and only if $K$ is pure.
\end{theorem}

\begin{remark}
If $K$ has no ghost vertices, then $\Phi_K$ is irreducible in $\bbC[x_1,\ldots,x_m,y_1,\ldots,y_r]$.
Indeed, the greatest common divisor of the monomials $g_{F_j}$, $F_j\in\Fac(K)$, is $1$ in this case, 
and $\Phi_K$ is linear in the variables $y_{F_j}$. 
Hence, any factorization of $\Phi_K$ would require one of the factors to divide every $g_{F_j}$, and thus to be a unit.
\end{remark}

By constructing suitable moment-angle complexes admitting nontrivial higher Massey products, and applying the realization in Theorem~\ref{thm:Suciu-realization}, Suciu obtained the following result.
This result provides the starting point for our work.

\begin{theorem}[{\cite[Theorem~4.1]{Suciu2026}}]
For every $k\geq 1$ and $n\geq 3$, there exists a weighted homogeneous polynomial with trivial geometric monodromy such that its Milnor fiber is $(2k+2)$-connected and admits a nontrivial $n$-fold Massey product.
\end{theorem}

Applying Theorem~\ref{thm:Suciu-realization} with $K$ replaced by $K(J)$, 
together with Theorem~\ref{thm:exponentiation}, yields the following corollary.

\begin{corollary}\label{thm:general-realization}
Let $K$ be a simplicial complex on the set $[m]$ and let $J=(j_1,\ldots,j_m)\in\bbZ_{>0}^m$.
Let $\Phi_{K(J)}$ be the polynomial given by~\eqref{eq:Suciu-poly} with $K$ replaced by $K(J)$.
Then $\Phi_{K(J)}$ is weighted homogeneous, has trivial geometric monodromy, and its Milnor fiber
$\Phi_{K(J)}^{-1}(1)$ is homotopy equivalent to $\mathcal Z_K\bigl((D^{2j_i},S^{2j_i-1})_{i=1}^m\bigr)$.
\end{corollary}


\subsection{The Kato--Matsumoto defect}\label{subsec:KM-background}

The dimension of the singular set $\operatorname{Sing}(\Phi_K^{-1}(0))$ of $\Phi_K^{-1}(0)$ is
determined by the combinatorial data of the simplicial complex $K$,
as shown in~\cite[Proposition~5.2]{Suciu2026}.
Put
\begin{equation}\label{eq:kappa}
   \kappa(K)=\min_{Z\notin K}\bigl(|Z|+\rho_K(Z)\bigr), 
\end{equation}
where $Z\subseteq \VVert(K)$ and
\begin{equation}\label{eq:rho}
   \rho_K(Z)
   =\#\{i\in Z \mid Z\setminus\{i\}\in K\}
   =\left|\bigcap_{\substack{I\in\MF(K)\\ I\subseteq Z}}I\,\right|.
\end{equation}

Recall that  $\Phi_K:\bbC^{m+r}\to\bbC$, where $m=|\VVert(K)|$ and $r=|\Fac(K)|$.
Setting $N=m+r$, we may write $\Phi_K:\bbC^N\to\bbC$.
The above $\kappa(K)$ is the codimension of the singular set of $\Phi_K^{-1}(0)$ in $\bbC^N$, as follows from~(1) below.

\begin{proposition}[{\cite[Proposition~5.2]{Suciu2026}}]\label{thm:Suciu-singular}
The following statements hold:
\begin{itemize}
\item[(1)] $\dim \operatorname{Sing}(\Phi_K^{-1}(0))=N-\kappa(K)$;
\item[(2)] $\kappa(K)\le 2\nu(K)$.
\end{itemize}
\end{proposition}

Kato and Matsumoto proved in~\cite{KM75} that if the singular set of a germ $f:(\bbC^N,0)\to(\bbC,0)$ has dimension $s$, then its Milnor fiber is $(N-s-2)$-connected. 
Note that the Kato--Matsumoto theorem uses the dimension of the singular set at the origin,
whereas $\dim \operatorname{Sing}(\Phi_K^{-1}(0))$ above refers to the dimension of an algebraic variety in $\bbC^N$.
However, these two dimensions coincide since $\Phi_K$ is weighted homogeneous.
Therefore, setting $s=\dim \operatorname{Sing}(\Phi_K^{-1}(0))$, 
the Kato--Matsumoto theorem gives the lower bound $\kappa(K)-2$ for the connectivity of the global Milnor fiber $\Phi^{-1}_K(1)$, whereas Proposition~\ref{prop:connectivity} gives the exact connectivity $2\nu(K)-2$.
Proposition~\ref{thm:Suciu-singular}(2) above follows from this observation.

\begin{definition}\label{def:KMdefect}
We call 
\[
   \delta_{\mathrm{KM}}(K)=2\nu(K)-\kappa(K)\ge 0
\]
the {\it Kato--Matsumoto defect} of $K$.
\end{definition}

The defect $\delta_{\mathrm{KM}}(K)$ is equal to the difference of 
the connectivity of the Milnor fiber $\Phi_K^{-1}(1)$ 
and the Kato--Matsumoto bound $\kappa(K)-2$.
Proposition~\ref{thm:Suciu-singular}(2) implies that the Kato--Matsumoto bound is sharp if and only if $\delta_{\mathrm{KM}}(K)=0$.

\section{Purity of complexes under simplicial multiwedges}\label{sec:purity}

Let $K$ be a simplical complex on the set $[m]$ and $K(J)$ be its simplicial multiwedge.
We prove the following proposition in this section.

\begin{proposition}\label{thm:degree-invariance}
Suppose that $K$ is pure with facets of cardinality $d$.  For every $J\in\bbZ_{>0}^m$, the polynomial $\Phi_{K(J)}$ is homogeneous of degree
\[
   \deg\Phi_{K(J)}=m-d+1.
\]
In particular, this degree is independent of $J$.
\end{proposition}

Recall that $B_i=\{i1,\ldots,ij_i\}\subset V(J)$ denotes the block corresponding to $i\in \VVert(K)$.
For a simplex $\hat \sigma\subseteq V(J)$ of $K(J)$,
we say that $B_i$ is {\it fully-contained} in $\hat\sigma$ if $B_i\subseteq\hat\sigma$.
Set
\[
   S(\hat \sigma)=\{i\in[m]: B_i\subseteq\hat\sigma\},
\]
which is the set of indices of blocks fully-contained in $\hat\sigma$.

\begin{lemma}\label{lem:membership}
A subset $\hat\sigma\subseteq V(J)$ is a simplex of $K(J)$ if and only if $S(\hat\sigma)$ is a simplex of $K$.
\end{lemma}

\begin{proof}
By definition, $\hat\sigma\notin K(J)$ if and only if it contains a minimal non-face $I(J)=\bigcup_{i\in I}B_i$ for some $I\in\MF(K)$.  This is equivalent to $I\subseteq S(\hat\sigma)$ for some $I\in\MF(K)$, which is equivalent to $S(\hat\sigma)\notin K$.
\end{proof}

The facets of the simplicial multiwedge are described as follows.

\begin{lemma}\label{prop:facets}
Let $F_j$ be a facet of $K$.  
Choose one vertex $b_i\in B_i$ for each $i\notin F_j$ and denote the set of these vertices by ${\bf b}_j$.
Set
\begin{equation}\label{eq:facet-lift}
   \widehat F_{j,{\bf b}_j}
   =\Bigl(\bigcup_{i\in F_j}B_i\Bigr)
   \cup
   \Bigl(\bigcup_{i\notin F_j}(B_i\setminus\{b_i\})\Bigr).
\end{equation}
Then $\widehat F_{j,{\bf b}_j}$ is a facet of $K(J)$, and every facet of $K(J)$ is defined uniquely in this way.

Consequently, it follows that
\begin{equation}\label{eq:facet-size}
   \card{\widehat F_{j,{\bf b}_j}}
   =\sum_{i=1}^m(j_i-1)+\card{F_j}
\end{equation}
and
\begin{equation}\label{eq:facet-number}
   \card{\Fac(K(J))}
   =\sum_{F_j\in\Fac(K)}\prod_{i\notin F_j}j_i.
\end{equation}
\end{lemma}

\begin{proof}
From~\eqref{eq:facet-lift}, we obtain $S(\widehat F_{j,{\bf b}_j})=F_j$, since
the indices of the fully-contained blocks of $\widehat F_{j,{\bf b}_j}$ are exactly the vertices of $F_j$.
Therefore, $\widehat F_{j,{\bf b}_j}$ is a simplex of $K(J)$ by Lemma~\ref{lem:membership}.
Furthermore, $\widehat F_{j,{\bf b}_j}$ is maximal. Indeed, if an omitted vertex $b_i$, $i\not\in F_j$, is added to $\widehat F_{j,{\bf b}_j}$, then $S(\widehat F_{j,{\bf b}_j}\cup\{b_i\})=F_j\cup\{i\}$, which is not a simplex since $F_j$ is a facet.

Conversely, let $\hat \sigma$ be a facet of $K(J)$.  
By Lemma~\ref{lem:membership}, $S(\hat\sigma)\in K$. 
If $S(\hat\sigma)$ is not a facet, 
there exists $i\notin S(\hat\sigma)$ such that $S(\hat\sigma)\cup\{i\}\in K$.
Since the corresponding block $B_i$ is not fully-contained in $\hat\sigma$, adding all the missing vertices of $B_i$ would produce a larger simplex of $K(J)$, contradicting the maximality of $\hat\sigma$.
Thus, $S(\hat\sigma)$ is a facet of $K$.

If, for some $i\notin \sigma$, at least two vertices of $B_i$ do not belong to $\hat\sigma$, then one of them could be added to $\hat\sigma$ without making $B_i$ fully-contained in $\hat\sigma$, 
and hence without changing $S(\hat\sigma)$.
This would again contradict the maximality of $\hat\sigma$.  
Therefore exactly one vertex is omitted from every block $B_i$ with $i\notin \sigma$, while each block $B_i$ with $i\in \sigma$ is fully-contained in $\hat\sigma$. 
This proves the description and its uniqueness.

From~\eqref{eq:facet-lift}, we obtain~\eqref{eq:facet-size},
and summing the number of independent choices of 
the omitted vertex $b_i$ over all facets $F_j$ gives \eqref{eq:facet-number}.
\end{proof}

\begin{corollary}\label{cor:purity}
The complex $K(J)$ is pure if and only if $K$ is pure.
\end{corollary}

\begin{proof}
By \eqref{eq:facet-size}, the cardinality of a facet of $K(J)$ differs from the cardinality of the corresponding facet of $K$ by the constant $\sum_{i=1}^m(j_i-1)$.
\end{proof}

\begin{lemma}\label{eq:complement-size}
$\card{V(J)\setminus\widehat F_{j,{\bf b}_j}}=m-|F_j|$ for each $j=1,\ldots,r$ and ${\bf b}_j$.
In particular, this cardinality is independent of $J$
and of the choice of $b_i$ for $i\not\in F_j$ in Lemma~\ref{prop:facets}.
\end{lemma}

\begin{proof}
The assertion follows from~\eqref{eq:facet-size} and the fact $|V(J)|=\sum_{i=1}^mj_i$.
\end{proof}

Now we prove Proposition~\ref{thm:degree-invariance}.

\begin{proof}[Proof of Proposition~\ref{thm:degree-invariance}]
By Corollary~\ref{cor:purity}, $K(J)$ is pure, so Theorem~\ref{thm:Suciu-realization} shows that $\Phi_{K(J)}$ is homogeneous. 
Every term of $\Phi_{K(J)}$ has the form
\[
   y_{\widehat F}\prod_{ij_i\notin \widehat F}x_{ij_i},
\]
where the variables $(x_{i1},\ldots,x_{mj_m})$ correspond to the set $V(J)$ of vertices of $K(J)$
and the variable $y_{\widehat F}$ corresponds to a facet $\widehat F$ of $K(J)$.
Note that the number of variables $y_{\widehat F}$ is given in~\eqref{eq:facet-number}.
The number of the variables $x_{ij_i}$ appearing in the above product is $|V(J)|-\card{\widehat F}$, 
where $|\widehat F|$ is given in~\eqref{eq:facet-size}.
Therefore, it follows from Lemma~\ref{eq:complement-size} that
the degree of each term of $\Phi_{K(J)}$ is 
\[
   |V(J)|-\card{\widehat F}+1
   =|V(J)\setminus \widehat F|+1
   =m-|S(\widehat F)|+1=m-d+1.
\]
Thus the assertion follows.
\end{proof}

\section{The Kato--Matsumoto defect under simplicial multiwedges}\label{sec:KMmultiwedge}

\subsection{Defects of multiwedges}

Let $K$ be a simplicial complex on $[m]$ and $J\in\bbZ_{>0}^m$. For $Z\subseteq \VVert(K)$, set
\[
   w_J(Z)=\sum_{i\in Z}j_i,
   \qquad
   \nu_J(K)=\min_{I\in\MF(K)}w_J(I).
\]
It follows from the definition of $K(J)$ that $\nu_J(K)=\nu(K(J))$. 
For a nonempty subfamily $\mathcal A\subseteq\MF(K)$,
set
\[
   \mathcal U(\mathcal A)=\bigcup_{I\in\mathcal A}I,
   \qquad
   \mathcal I(\mathcal A)=\bigcap_{I\in\mathcal A}I.
\]

\begin{theorem}\label{thm:union-intersection}
For every non-simplex simplicial complex $K$ on the set $[m]$ and every $J\in\bbZ_{>0}^m$,
\begin{equation}\label{eq:union-intersection}
\kappa(K(J))
   =\min_{\varnothing\ne\mathcal A\subseteq\MF(K)}
      \bigl(w_J(\mathcal U(\mathcal A))+w_J(\mathcal I(\mathcal A))\bigr),
\end{equation}
and therefore
the Kato--Matsumoto defect $\delta_{\mathrm{KM}}(K(J))$ is given by
\[
   \delta_{\mathrm{KM}}(K(J))
   =2\nu_J(K)-\kappa(K(J)).
\]
In particular, 
every positive defect is caused by at least three minimal non-faces.
\end{theorem}

\begin{proof}
Let $\widehat Z$ be a non-face of $K(J)$,
and $S(\widehat Z)$ be the index set of fully-contained blocks of $\widehat Z$, that is,
$S(\widehat Z)=\{i\in[m] \mid B_i\subseteq \widehat Z\}$.
By the definition of $K(J)$,
the minimal non-faces contained in $\widehat Z$ are precisely 
$B_I=\bigcup_{i\in I}B_i$ for $I\in\MF(K)$ with $I\subseteq S(\widehat Z)$. Hence \eqref{eq:rho} gives
\begin{equation}\label{eq:rho-multiwedge}
   \rho_{K(J)}(\widehat Z)
   =w_J\left(\bigcap_{\substack{I\in\MF(K)\\ I\subseteq S(\widehat Z)}}I\,\right). \end{equation}
Deleting vertices of $\widehat Z$ belonging to a block that is not fully-contained in $\widehat Z$
affects neither the minimal non-faces $I$ in the brackets on the right-hand side of~\eqref{eq:rho-multiwedge} nor their intersections, while decreasing  $|\widehat Z|$.
Thus a minimizer in~\eqref{eq:kappa} may be chosen to be a union of blocks fully-contained in $\widehat Z$.

For $W\subseteq\VVert(K)$, put
\[
   \mathcal A_W=\{I\in\MF(K)\mid I\subseteq W\}.
\]
If $\mathcal A_W\ne\varnothing$, replacing $W$ by $\mathcal U(\mathcal A_W)$ changes neither $\mathcal A_W$ nor the intersections of sets in $\mathcal A_W$, while possibly decreasing $w_J(W)$. Consequently,
\[
   \kappa(K(J))=
   \min_{W:\mathcal A_W\ne\varnothing}
   \bigl(w_J(W)+w_J(\mathcal I(\mathcal A_W))\bigr)
\]
is attained with $W=\mathcal U(\mathcal A_W)$. 
Here we used the equality $\rho_{K(J)}(W^{full})=w_J(W)$, where $W^{full}$ is the union of the blocks corresponding to the vertices of $W$, together with the fact that a minimizer in~\eqref{eq:kappa} may be chosen to be a union of blocks, as explained in the first paragraph.
This shows that the right-hand side of \eqref{eq:union-intersection} is at most $\kappa(K(J))$.

Conversely, let $\varnothing\ne\mathcal A\subseteq\MF(K)$ and set $W=\mathcal U(\mathcal A)$. The family $\mathcal A_W$ may be larger than $\mathcal A$, but then $\mathcal I(\mathcal A_W)\subseteq \mathcal I(\mathcal A)$. Hence
\[
   \kappa(K(J))
   \le w_J(W)+w_J(\mathcal I(\mathcal A_W))
   \le w_J(\mathcal U(\mathcal A))+w_J(\mathcal I(\mathcal A)).
\]
Taking the minimum over $\mathcal A$ gives the reverse inequality.

To prove the last assertion, it suffices to show that neither a single minimal non-face nor a pair of minimal non-faces can cause a positive defect.
For a singleton $\mathcal A=\{I\}$, the value is $2w_J(I)\ge2\nu_J(K)$. For $\mathcal A=\{I_1,I_2\}$, finite additivity gives
\[
   w_J(I_1\cup I_2)+w_J(I_1\cap I_2)
   =w_J(I_1)+w_J(I_2)\ge2\nu_J(K).
\]
This completes the proof.
\end{proof}

\subsection{Flag complexes and weighted triangles}\label{sec:flagKM}

A simplicial complex is called a {\it flag} complex if
every set of vertices that are pairwise connected by edges spans a simplex of the complex.
Suppose that a simplicial complex $K$ is flag and not a simplex. 
Its minimal non-faces are edges. 
Let $\Edge(G)$ denote the set of edges of a graph $G$. An edge connecting vertices $u$ and $v$ is denoted by $uv$. 

Let $H$ be the {\it missing-edge graph} of $K$, whose edges connect the pairs of vertices ${u,v}$ 
that are not simplices of $K$, that is,
\[
   uv\in \Edge(H)\quad\Longleftrightarrow\quad \{u,v\}\notin K.
\]
Note that $K$ coincides with the independence complex of the graph $H$.
Indeed, since the edges of $H$ are precisely the missing edges of the flag complex $K$, a subset $\sigma$ is a simplex of $K$ if and only if it contains no edge of $H$.

Set
\[
\begin{split}
   \varepsilon_J(H)&=\min\{w_J(e)\mid \text{$e$ is an edge of $H$}\}, \\
   \tau_J(H)&=\min\{w_J(T)\mid \text{$T$ is a triangle of $H$}\},
\end{split}
\]
with $\tau_J(H)=+\infty$ if $H$ is triangle-free.

\begin{theorem}\label{thm:weighted-triangle}
For every flag complex $K$ and every $J\in\bbZ_{>0}^m$, the following statements hold:
\begin{itemize}
\item[(1)] $\kappa(K(J))=\min\{2\varepsilon_J(H),\tau_J(H)\}$; 
\item[(2)] $\delta_{\mathrm{KM}}(K(J))=\max\{0,2\varepsilon_J(H)-\tau_J(H)\}$.
\end{itemize}
\end{theorem}

\begin{proof}
We apply Theorem~\ref{thm:union-intersection} to a nonempty family $\mathcal A$ of edges of $H$. If all edges share a common endpoint $c$, then either $\mathcal A$ is a singleton $ca$ or it contains distinct edges $ca$ and $cb$.
In the former case, $w_J(U(\mathcal A))+w_J(I(\mathcal A))=2w_J(ca)\ge2\varepsilon_J(H)$.
In the latter case,
\[
   w_J(U(\mathcal A))+w_J(I(\mathcal A))
   \ge2j_c+j_a+j_b
   =w_J(ca)+w_J(cb)\ge2\varepsilon_J(H).
\]
If there is a pair of disjoint edges $ab$ and $cd$ in $\mathcal A$, then $w_J(ab)+w_J(cd)$ is already at least $2\varepsilon_J(H)$. 
Thus, it remains to consider a pairwise-intersecting edge family with no common endpoint. 
Choose $ab,ac\in\mathcal A$. Any edge not containing $a$ but intersecting both $ab$ and $ac$ must be $bc$.
Any other edge intersecting $ab$, $ac$, and $bc$ must be one of these three. 
Hence $\mathcal A$ consists of the edges of a triangle, and the contribution to $\kappa(K(J))$ in~\eqref{eq:union-intersection} is at least $j_a+j_b+j_c$ for a triangle $\{a,b,c\}$ in $H$, which is at least
$\tau_J(H)$.
Thus, assertion~(1) follows. Assertion~(2) follows from Definition~\ref{def:KMdefect},
together with $\varepsilon_J(H)=\nu_J(K)$.
\end{proof}


The Kato--Matsumoto estimate for $K(J)$ is sharp if and only if
\[
   w_J(T)\ge2\min_{e\in \Edge(H)}w_J(e)
\]
for every triangle $T\subset H$. Its positive sharpness locus is the finite union of rational polyhedral cones
\[
   C(K)=\bigcup_{e\in \Edge(H)}C_e,
\]
where
\[
   C_e=\{J\in \bbZ_{>0}^m \mid w_J(e)\le w_J(e')\ \forall e'\in \Edge(H) \;\text{and}\;
   2w_J(e)\le w_J(T) \ \forall T\subset H\}.
\]
Thus, the sharp bound holds for almost every $m$-tuple $J$.
The following corollary characterizes those simplicial complexes $K$ satisfying $C(K)=\bbZ_{>0}^m$.

\begin{corollary}\label{cor:flag-universal}
Let $K$ be a flag complex with missing-edge graph $H$.  Then the following are equivalent.
\begin{enumerate}
\item $\delta_{\mathrm{KM}}(K(J))=0$ for every positive multiwedge $J$;
\item $H$ is triangle-free.
\end{enumerate}
If $J=(k,\ldots,k)$ with $k\geq 1$ and $H$ contains a triangle, then $\delta_{\mathrm{KM}}(K(J))=k$.
\end{corollary}

\begin{proof}
If $H$ is triangle-free, then $\tau_J(H)=+\infty$, and hence Theorem~\ref{thm:weighted-triangle} 
implies~(1).
Conversely, if $H$ contains a triangle and $J=(k,\ldots, k)$, then $\varepsilon_J(H)=2k$ and $\tau_J(H)=3k$, so Theorem~\ref{thm:weighted-triangle}(2) gives $\delta_{\mathrm{KM}}(K(J))=k>0$.
\end{proof}

\section{Fixed-degree homogeneous non-formal Milnor fibers}\label{sec:main}

\subsection{Limonchenko's simplicial multiwedges}\label{sec:limonchenko}

We recall Limonchenko's $n$-dimensional $2$-truncated cube $P(n)$ introduced in~\cite[Definition~3.5]{Limonchenko2019}.
For a polytope $P$, let $K_P$ denote its {\it nerve complex}, that is, the simplicial complex whose vertices correspond to the facets of $P$, and whose simplices are the collections of facets with nonempty intersection.

For $n\geq 2$,
let $K(n)$ be the simplicial complex on the set $[2n]$ with
\[
   \MF(K(n))=\{(k, n+k+i)\mid 0\leq i\leq n-2,\; 1\leq k\leq n-i\}.
\]   
Let $Q(n)$ be the $n$-dimensional cube $[0,1]^n$, with facets $F_1,\ldots,F_{2n}$ labelled so that, for $i=1,\ldots,n$,
$F_i$ contains the origin of $\bbR^n$ and the inward normal vector of $F_i$ is the $i$-th standard basis vector,
while $F_{n+i}$ is parallel to $F_i$, see Figure~\ref{fig1}.
Then, 
define $P(n)$ to be the simple polytope with 
\begin{equation}\label{eq:mr}
   m=\frac{n(n+3)}2-1
\end{equation}
facets, which is obtained from $Q(n)$ by truncating the codimension-two faces so that
the induced subcomplex of $K_{P(n)}$ on the vertex set $[2n]$ is combinatorially equivalent to $K(n)$ 
(cf.~\cite[Definition~4.1]{Lim17}).
More precisely,
the truncations are applied to the faces $F_p\cap F_q$, where $\{p,q\} \in \MF(K(n))$.
Therefore, the number of truncations is $n(n-1)/2-1$.
Since $Q(n)$ has $2n$ facets, the number of facets of $P(n)$ is 
\[
2n+\frac{n(n-1)}{2}-1=\frac{n(n+3)}{2}-1=m,
\]
as claimed in~\eqref{eq:mr}.

\begin{example}\label{example51}
The polytope $P(3)$ is shown in Figure~\ref{fig1}.
In this case, 
\[
   \MF(K(3))=\{\{1,4\}, \{2,5\}, \{3,6\}, \{1,5\}, \{2,6\}\},
\]
among which $F_p\cap F_q\neq \varnothing$ only when $\{p,q\}=\{1,5\}$ or $\{2,6\}$.
Therefore, $P(3)$ is obtained from the cube $Q(3)$ by truncating the edges $F_1\cap F_5$ and 
$F_2\cap F_6$. This is the polyhedron in Figure~\ref{fig1}.
Since $P(3)$ has $8$ facets, $\VVert(K_{P(3)})=8$.
The missing-edge graph $H$ of $K_{P(3)}$ consists of the edges in $\MF(K(3))$, 
and therefore, $H$ is triangle-free.
Hence, $\delta_{KM}(K(J))=0$ holds for any $J\in\bbZ_{>0}^8$ by Corollary~\ref{cor:flag-universal}.
\end{example}

\begin{figure}[htbp]
\includegraphics[scale=0.8, bb=174 557 398 712]{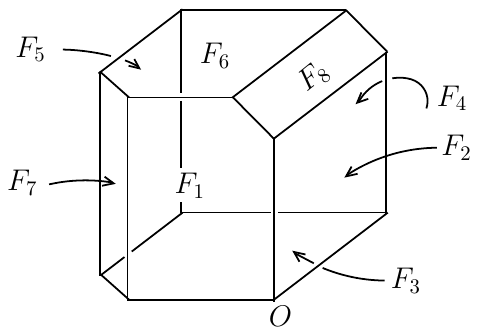}
\caption{Limonchenko's polytope $P(3)$.}
\label{fig1}
\end{figure}

The nerve complex $K_{P(n)}$ of $P(n)$ is a flag polytopal $(n-1)$-sphere. 
Hence, it is pure and has no ghost vertices. 
Let $A=\{1,\ldots,n\}$ and $B=\{n+1,\ldots,2n\}$ be the sets of vertices of $K_{P(n)}$ 
corresponding to the facets $F_1,\ldots, F_n$ and $F_{n+1},\ldots, F_{2n}$ of $P(n)$, respectively.
The full subcomplex of $K_{P(n)}$ on $A\sqcup B$ is the complex $K(n)$
and each minimal non-face of $K_{P(n)}$ on $A\sqcup B$ joins a vertex in $A$ and a vertex in $B$.
In particular, the missing-edge graph of $K_{P(n)}$ induced on $A\sqcup B$ is bipartite.

Now we consider the simplicial multiwedge of $K_{P(n)}$.
Let $J$ be the $m$-tuple of positive integers in~\eqref{eq1-1}
such that $j_{p,n,s}=1$ for $p>2n$ if $s=1$ and 
$|I|\ge s+1$ for every minimal non-face $I$ of $K_{P(n)}(J)$.
This is exactly the freedom allowed in \cite[Definition~3.5]{Limonchenko2019}.  
Since $K_{P(n)}$ is flag, the condition $|I|\geq s+1$ is equivalent to the collection of inequalities
\[
   j_{p,n,s}+j_{q,n,s}\ge s+1
   \qquad\text{for every missing-edge\;}\{p, q\}\in\MF(K_{P(n)}),
\]
where $j_{i,n,s}=s$ for $1\le i\le n$ and $j_{i,n,s}=1$ for $n+1\le i\le2n$.
Set
\begin{equation}\label{eq:mns}
   d(J)
   =n(s+1)+\sum_{p=2n+1}^{m}j_{p,n,s},
\end{equation}
then we have $|\VVert(K_{P(n)}(J))|=d(J)$
and $\dim K_{P(n)}(J)=d(J)+n-m-1$.

The distinguished minimal non-faces of $K_{P(n)}(J)$ contained in the full subcomplex on the first $2n$ blocks have the form
\[
   \{i1, \ldots, is, n+i\},
   \qquad 1\le i\le n,
\]
and hence have cardinality exactly $s+1$.  Together with the condition $|I|\geq s+1$, 
we have
\begin{equation}\label{eq:nu-Krs}
   \nu(K_{P(n)}(J))=s+1.
\end{equation}
Note that the full subcomplex of $K_{P(n)}(J)$ on the first $2n$ blocks is precisely Limonchenko's $K(n,s)$,
where
\begin{equation}\label{eq:distinguished-Krs}
   K(n,s)=K(n)(\underbrace{s,\ldots,s}_{n},
                   \underbrace{1,\ldots,1}_{n}).
\end{equation}

\subsection{The Kato--Matsumoto bound}

For calculating the Kato--Matsumoto bound, we need one additional condition.
Let $H_n$ denote the missing-edge graph of $K_{P(n)}$. Recall from
Subsection~\ref{sec:flagKM} that, for $J=(j_1,\ldots,j_m)\in\bbZ_{>0}^m$,
\[
\tau_J(H_n)
=
\min_{\{a,b,c\}\text{: a triangle of }H_n}
(j_a+j_b+j_c),
\]
with the convention $\tau_J(H_n)=+\infty$ if $H_n$ is triangle-free.

We say that a choice of the indices $j_{2n+1,n,s},\ldots, j_{m,n,s}$ in~\eqref{eq1-1} is
\emph{KM-admissible} if
\begin{equation}\label{eq:KM-admissible}
\tau_J(H_n)\ge 2(s+1).
\end{equation}
This condition is not part of Limonchenko's definition.  For every $s\ge2$, it can be achieved simultaneously with 
the condition $|I|\geq s+1$: the induced missing-edge graph on the first $2n$ vertices is bipartite, so every triangle of $H_n$ contains a vertex with index $p>2n$ and the corresponding integers $j_{p,n,s}$ in $J$ may be chosen sufficiently large. We do not impose \eqref{eq:KM-admissible} unless Kato--Matsumoto sharpness is explicitly asserted.

\begin{proposition}\label{prop:Limonchenko-kappa}
It follows that
\[
   \nu(K_{P(n)}(J))=s+1
   \quad\text{and}\quad
   \kappa(K_{P(n)}(J))
   =\min\bigl\{2(s+1),\tau_{J}(H_n)\bigr\}.
\]
Consequently, $\mathcal Z_{K_{P(n)}(J)}$ is exactly $2s$-connected.  The Kato--Matsumoto bound is sharp if and only if the chosen $J$ is KM-admissible in the sense of \eqref{eq:KM-admissible}; in that case $\kappa(K_{P(n)}(J))=2(s+1)$.
\end{proposition}

\begin{proof}
By the definition of $K_{P(n)}(J)$,
every minimal non-face of $K_{P(n)}(J)$ has cardinality at least $s+1$, while the distinguished minimal non-faces in \eqref{eq:distinguished-Krs} have cardinality exactly $s+1$. Hence $\nu(K_{P(n)}(J))=s+1$,
which proves the first assertion.
Consequently, $\varepsilon_J(H_n)=s+1$.
Hence, the second assertion follows from Theorem~\ref{thm:weighted-triangle}(1),
applied to the flag complex $K_{P(n)}$ with the $m$-tuple $J$.
The exact connectivity follows from Proposition~\ref{prop:connectivity}. 
The final assertion is immediate from Definition~\ref{def:KMdefect} and \eqref{eq:KM-admissible}.
\end{proof}

\subsection{Proof of Theorem~\ref{thm1}}

The assertion concerning Massey products in Theorem~\ref{thm1} follows directly from Limonchenko's construction in~\cite{Limonchenko2019}.
The relevant background and basic notions concerning formality and Massey products of moment-angle complexes were briefly reviewed in Subsection~\ref{sec2-3}. 
We refer to~\cite{DS07, Lim17, Limonchenko2019, GL21, Suciu2026}
and the references therein for further details.

By~\cite[Theorem~4.1]{Limonchenko2019}, applied with
$\mathbf{k}=\mathbb{Q}$ and $r=n$, the rational cohomology algebra
$H^*(\mathcal Z_{K_{P(n)}(J)};\mathbb Q)$ contains a strictly defined nontrivial $n$-fold Massey product.
The input classes are the Hochster classes associated to the distinguished
minimal non-faces of cardinality $s+1$, and hence have degree $2(s+1)-1=2s+1$.
Therefore, the $n$-fold product lies in degree
$n(2s+1)-(n-2)=2ns+2$.
In particular, the moment-angle complex $\mathcal Z_{K_{P(n)}(J)}$ is 
non-formal over $\bbQ$.


Now we prove Theorem~\ref{thm1}.

\begin{proof}[Proof of Theorem~\ref{thm1}]
Fix integers $n\ge 3$ and $s\ge1$.
Let $K_{P(n)}(J)$ be the simplicial multiwedge of $K_{P(n)}$ with $J$ introduced in Section~\ref{sec:limonchenko},
where $J$ is the $m$-tuple in~\eqref{eq1-1}.
Since $P(n)$ is a simple polytope, its nerve complex $K_{P(n)}$ is pure,
and therefore its simplicial multiwedge $K_{P(n)}(J)$ is also pure by Corollary~\ref{cor:purity}.
Then the weighted homogeneous polynomial $\Phi_{K_{P(n)}(J)}(x,y)$ obtained from $K_{P(n)}(J)$
in Theorem~\ref{thm:Suciu-realization},
which has trivial geometric monodromy and whose Milnor fiber is homotopy equivalent to $\mathcal Z_{K_{P(n)}(J)}$,
is actually homogeneous.
We set $f_{n,J}=\Phi_{K_{P(n)}(J)}$.

The complex $K_{P(n)}$ has $m$ vertices, where $m$ is given in \eqref{eq:mr}, and its facets have cardinality $n$. 
Therefore, Proposition~\ref{thm:degree-invariance} gives
\[
   \deg f_{n,J}
   =m-n+1
   =\left(\frac{n(n+3)}2-1\right)-n+1
   =\frac{n(n+1)}2.
\]

By \eqref{eq:nu-Krs}, $\nu(K_{P(n)}(J))=s+1$.  Thus, it follows from Proposition~\ref{prop:connectivity} 
that the Milnor fiber $f_{n,J}^{-1}(1)$ of $f_{n,J}$ is $2s$-connected and this connectivity is exact. This proves~(1).

By \cite[Theorem~4.1]{Limonchenko2019}, applied with
$\mathbf{k}=\mathbb Q$ and $r=n$, the rational cohomology
$H^*(\mathcal Z_{K_{P(n)}(J)};\mathbb Q)$ contains a strictly defined
nontrivial $n$-fold Massey product. 
As mentioned at the beginning of this subsection,
the input classes have degree $2s+1$ and the product lies in degree $2ns+2$.
Transporting the product across the 
homotopy equivalence between $f_{n,J}^{-1}(1)$ and $\mathcal Z_{K_{P(n)}(J)}$ proves~(2), and non-formality follows from the existence of a nontrivial Massey product.

Finally, Proposition~\ref{prop:Limonchenko-kappa} gives
\[
   \operatorname{codim}\operatorname{Sing}(f_{n,J})
   =\min\bigl\{2(s+1),\tau_{J}(H_n)\bigr\}.
\]
For a KM-admissible choice, this equals $2s+2$, so the Kato--Matsumoto lower bound is $2s$ and is sharp.  This proves the last assertion.
\end{proof}

\begin{remark}\label{rem:ambient}
From \eqref{eq:mns} and \eqref{eq:facet-number},
the ambient dimension $N$ of the homogeneous polynomial $f_{n,J}:\bbC^N\to\bbC$ in Theorem~\ref{thm1} is
\[
   N=d(J)+|\Fac(K_{P(n)}(J))|,
\]
where
\[
   |\Fac(K_{P(n)}(J))|
   =\sum_{F\in\Fac(K_{P(n)})}\prod_{i\notin F}j_{i,n,s}.
\]
Therefore, the ambient dimension $N$ grows with $s$, whereas the ordinary polynomial degree $\deg f_{n,J}$ remains the fixed value $n(n+1)/2$.
\end{remark}

\subsection{Applications to Suciu's problems}\label{sec:applications}


As mentioned in the introduction, 
Theorem~\ref{thm1} provides homogeneous polynomials that resolve~\cite[Problem~6.1]{Suciu2026},
with the additional property that their Kato--Matsumoto estimates are sharp.
Note that the homogeneous realization not only exists, but can be chosen with a degree depending only on the order of the Massey product and not on the connectivity $2s$.

\begin{example}
Set $n=3$ and $J=(s,s,s,1,1,1,j_{7,3,s},j_{8,3,s})$ in Theorem~\ref{thm1}.
Then, the degree of $f_{3,J}$ is $n(n+1)/2=6$.
Thus, for every $s\geq 2$, there exists a homogeneous sextic polynomial $f_{3,J}$ with trivial geometric monodromy whose Milnor fiber is exactly $2s$-connected and has
a strictly defined nontrivial Massey product in degree $6s+2$.
Any choice of the indices $(j_{7,3,s},j_{8,3,s})$ attains the Kato--Matsumoto bound,
as shown in Example~\ref{example51}.
\end{example}


The same realization applies to the generalized moment-angle complex $\mathcal Z_K\bigl((D^{2j_i},S^{2j_i-1})_{i=1}^m\bigr)$
as stated in Theorem~\ref{thm:exponentiation}.
The following corollary, with $J=(j_1,\ldots,j_m)=(k,\ldots,k)$ for $k\geq 2$, 
provides an answer to~\cite[Problem~6.6]{Suciu2026} while preserving all the properties in Theorem~\ref{thm1}.

\begin{corollary}\label{cor:uniform-realization}
For every simplicial complex $K$ with $m$ vertices and every $m$-tuple $J=(j_1,\ldots,j_m)$ of positive integers, there exists a weighted homogeneous polynomial $\Phi_{K,J}$ with trivial geometric monodromy such that the Milnor fiber
$\Phi_{K,J}^{-1}(1)$ is homotopy equivalent to $\mathcal Z_K\bigl((D^{2j_i},S^{2j_i-1})_{i=1}^m\bigr)$. 
Moreover, for $K=K_{P(n)}$, $\Phi_{K,J}$ is homogeneous and satisfies all the properties asserted in Theorem~\ref{thm1} for each pair of $n\geq 3$ and $s\geq 2$.
\end{corollary}

\begin{proof}
By~Corollary~\ref{thm:general-realization}, $\Phi_{K,J}=\Phi_{K(J)}$ is the polynomial appearing in the first assertion.
If $K$ is pure, then
$\mathcal Z_K\bigl((D^{2j_i},S^{2j_i-1})_{i=1}^m\bigr)$ is homotopy equivalent to the Milnor fiber $\Phi_{K(J)}^{-1}(1)$ of a homogeneous polynomial $\Phi_{K(J)}$. 

Now set $K=K_{P(n)}$. Since $K_{P(n)}$ is pure, $K_{P(n)}(J)$ is also pure by Corollary~\ref{cor:purity}.
Therefore, Theorem~\ref{thm:Suciu-realization}, applied with $K=K_{P(n)}$,
implies that $\Phi_{K(J)}$ is homogeneous.
The construction of $\Phi_{K(J)}$ is exactly the same as in the proof of Theorem~\ref{thm1},
so that the last assertion follows.
\end{proof}

Note that, in the above proof, if $K$ is pure, then $\deg \Phi_{K(J)}=|\VVert(K)|-|\Fac(K)|+1$ by 
Proposition~\ref{thm:degree-invariance}.
In particular,  $\deg\Phi_{K(J)}$ is independent of $J$.

\end{document}